\documentclass[a4paper]{article}

\usepackage[T1]{fontenc}
\usepackage{amsmath, amsfonts, amsthm, amssymb}
\usepackage[shortlabels]{enumitem}
\usepackage{bm}
\usepackage{microtype}
\usepackage[numbers,square]{natbib}

\usepackage{hyperref}

\newtheorem{lemma}{Lemma}
\newtheorem{theorem}[lemma]{Theorem}
\newtheorem{conjecture}[lemma]{Conjecture}
\newtheorem{proposition}[lemma]{Proposition}

\theoremstyle{definition}
\newtheorem{definition}[lemma]{Definition}

\newcommand{\bigmid}{\; \big | \;}

\title{Product-free sets in the free group}
 \author{Miquel Ortega \and Juanjo Ru\'e \and Oriol Serra \thanks{Departament de Matem\`atiques and Institut de Matem\`atiques (IMTech) de la Universitat Polit\`ecnica de Catalunya (UPC), and Centre de Recerca Matem\`atica (CRM), Barcelona, Spain. 
  E-mail: \texttt{miquel.ortega.sanchez-colomer@upc.edu}, \texttt{juan.jose.rue@upc.edu},
  \texttt{oriol.serra@upc.edu.}
 }}
 \date{}

\begin{document}
\maketitle
\begin{abstract}
        We prove that product-free subsets of the free group over a finite alphabet
        have maximum upper density $1/2$ with respect to the natural measure that
        assigns total weight one to each set of irreducible words of a given
        length. This confirms a conjecture of Leader, Letzter, Narayanan and
        Walters. In more general terms, we actually prove that strongly
        $k$-product-free sets have maximum upper density $1/k$ in terms of this
        measure. The bounds are tight.
\end{abstract}

\section{Introduction}
A subset $S$ of a group is said to be \emph{product-free} if there do not exist
$x, y, z \in S$ (not necessarily distinct) such that $z = x\cdot y$. Much has
been studied about product-free subsets of finite groups, particularly so in the
abelian case, where they are usually called \emph{sum-free subsets} (see, for example,
the survey by Tao and Vu \cite{2017.TV}). The study of product--free subsets in
nonabelian groups can be traced back to Babai and S\'os \cite{1985.BS}, see the
survey by Kedlaya \cite{2009.Kedlaya}. Interest on the problem was prompted by
the seminal work of Gowers on quasirandom groups \cite{2008.Gowers}.  

The study of product-free sets in discrete infinite structures is more recent. 
As a first approach to the study of the infinite case, Leader, Letzter, Narayanan,
and Walters in 
\cite{2020.LLNW} proved that product-free subsets of the free \emph{semigroup}
on the finite alphabet $\mathcal{A}$ have maximum density $1/2$ with respect to the
measure that assigns a weight of $|\mathcal{A}|^{-n}$ to every word of length
$n$. They conjectured that this is also true for the analogous measure on the
free group. The main purpose of the present paper is to provide a proof of this conjecture. 

More precisely, let us write $\mathcal{F}$ for the free group over a finite
alphabet $\mathcal{A}$. For $n \geq 1$ and $A \subseteq \mathcal{F}$, we write
$A(n) = \left\{ w \in A \colon |w| = n \right\}$, where $|w|$ stands for the
length of the reduced word of $w$, and $A_{\leq n}$ for those that have length
smaller or equal than $n$. We define a measure $\mu$ on $\mathcal{F}$ such
that
\begin{equation}
\label{eq:definition_mu}
\mu(\left\{ w  \right\}) = \frac{1}{|\mathcal{F}(|w|)|} =
\frac{1}{2|\mathcal{A}|(2|\mathcal{A}|-1)^{|w|-1}}
\end{equation}
for all $w \in \mathcal{F}$, so that every layer of words of a given
length has the same total weight. Finally, we write $\bar{d}(A) = \limsup_{n \to
\infty} \frac{\mu(A_{\leq n})}{\mu(\mathcal{F}_{\leq n})}$
for the upper asymptotic density of $A$. Our main result can be phrased as follows.

\begin{theorem}
        \label{theo:free_group}
        Let $S \subseteq \mathcal{F}$ be a product-free subset. Then
        \begin{equation}
                \bar{d}(S) \leq \frac{1}{2}.
        \end{equation}
\end{theorem}

We actually study a generalisation of this result. Following the notion by
\L uczak and Schoen \cite{2001.LS} we call a subset $S$ of a semigroup
\emph{$k$-product-free} for $k \geq 2$ if there are no $x_1, \dots, x_k, y
\in S$ such that $x_1 \cdots x_k = y$ and, furthermore, we call $S$
\emph{strongly} $k$-product-free if it is $\ell$-product-free for all $\ell$ with $2
\leq \ell \leq k$. We are able to prove the following.
\begin{theorem}
        \label{theo:kfree}
        Let $S \subseteq \mathcal{F}$ be a strongly $k$-product-free subset for
        $k \geq 2$. Then
        \begin{equation}
                \label{eq:bound_kfree}
                \bar{d}(S) \leq \frac{1}{k}.
        \end{equation}
\end{theorem}

The upper bound in \eqref{eq:bound_kfree} is best possible. Consider, for
example, an arbitrary letter $x \in \mathcal{A}$, and the set $S
\subseteq \mathcal{F}$ consisting of words such that the number of $x$ minus
the number of $x^{-1}$ in its reduced form is congruent to $1$ modulo $k$. This
set is strongly $k$-product-free and has upper asymptotic density $\bar{d}(S) =
1/k$.

The paper is organised as follows. In Section \ref{sec:reduction} we reduce the
problem to certain subsemigroups of the free group and state the key result,
Proposition \ref{prop:kfree_xy}, which provides the proof of Theorem
\ref{theo:kfree}. Section \ref{sec:semigroup} is devoted to the proof of
Proposition \ref{prop:kfree_xy}. For $2 \leq k \leq 3$, the proof is inspired by the arguments used
in \cite{2020.LLNW} for the case of semigroups, and we give a separate proof for this case for
clarity of exposition. 
The case of general $k$ is somewhat more involved. One must first use a density increment
argument to reach a pseudorandom subset of our given set. There, the general argument is
easier to carry out. Both arguments are presented in a probabilistic manner, which we believe gives
further insight to the ideas behind the proof. The paper concludes with some final
remarks in Section \ref{sec:final}. 

To conclude, we also want to remark that the statement
of Theorem \ref{theo:kfree} also holds
in the free semigroup, which actually was the model where we first
worked out the proof, and it is a generalisation of the main theorem in
\cite{2020.LLNW}.
\begin{theorem}\label{thm:semigroup}
         For any finite alphabet $\mathcal{A}$, a strongly
$k$-product-free of
         the free semigroup over $\mathcal{A}$ has upper asymptotic density at most
$1/k$.
\end{theorem}
The straightforward  translation of the arguments in the proof of
Theorem \ref{theo:kfree} to the case of semigroups is discussed in the
final section.


\section{Reduction to a semigroup}\label{sec:reduction}
Throughout the proof, we identify every element of $\mathcal{F}$ with its
reduced word (for example, the length of $w \in \mathcal{F}$ is the length of
its reduced word) and assume that $|\mathcal{A}| \geq 2$. If $|\mathcal{A}| =
1$, the free group is isomorphic to the integers, where the result is
straightforward.

We write $\mathcal{F}^{xy} \subseteq \mathcal{F}$ for the subset of words
that begin with $x$ and end in $y$ for $x, y \in \mathcal{A} \cup
\mathcal{A}^{-1}$, where $\mathcal{A}^{-1} = \left\{ x^{-1} \colon x \in
\mathcal{A} \right\}$. The first step of the proof consists in reducing the proof
of Theorem
\ref{theo:kfree} to an analogous statement over $\mathcal{F}^{xy}$ with $x \neq
y^{-1}$. This ambient space has the advantage of having no cancellation when
multiplying, so it is much closer to the case of the free semigroup, and we will
then be able to use ideas similar in spirit to those of \cite{2020.LLNW}.

For a given family $\mathcal{H} \subseteq \mathcal{F}$ and a subset $A \subseteq
\mathcal{H}$, we define the \emph{relative upper density} of $A$ in $\mathcal{H}$ as
\[
        \bar{d}_{\mathcal{H}}(A) = \limsup_{n \to
        \infty} \frac{\mu(A_{\leq n})}{\mu(\mathcal{H}_{\leq n})}.
\]
The analogous result to Theorem \ref{theo:kfree} then reads as follows.
\begin{proposition}
        \label{prop:kfree_xy}
        Given $x, y \in \mathcal{A} \cup \mathcal{A}^{-1}$ satisfying $x \neq
        y^{-1}$, let
        $S \subseteq \mathcal{F}^{xy}$ be a strongly $k$-product-free set with
        $k \geq 2$. Then
\[
        \bar{d}_{\mathcal{F}^{xy}}(S) \leq \frac{1}{k}.
\]
\end{proposition}
We will prove this result in the following section. In this section, we deduce
Theorem \ref{theo:kfree} from it.
\begin{proof}[Proof of Theorem \ref{theo:kfree} assuming Proposition
        \ref{prop:kfree_xy}]
        Let $\mathcal{I}$ be the set of tuples $(w, x, y)$, with $w \in
        \mathcal{F}$ and $x, y \in \mathcal{A} \cup \mathcal{A}^{-1}$ such that
        $wxyw^{-1}$ admits no cancellation. For $(w, x, y) \in \mathcal{I}$ write
        $S^{w,x,y} \subseteq S \subseteq \mathcal{F}$ for the elements of $S$ that may be written as
        $w x \alpha y w^{-1}$. Note that, since we imposed that $x \neq y^{-1}$,
        the elements of $S$ belong to a unique $S^{w,x,y}$.

        Let us define the map $\pi \colon S \to \mathcal{F}^{xy}$ that strips
        out $w$ and $w^{-1}$ from the elements of $S^{w, x, y}$, meaning that
        \[
                \pi(wx\alpha y w^{-1}) = x \alpha y
        \]
        for $wx \alpha y w^{-1} \in S^{w, x, y}$.
        In other words, the map $\pi$ sends an element of $S$ to
        its cyclically reduced word.

        Set a fixed $\varepsilon > 0$. Since $S$ is strongly $k$-product-free, so is $S^{w,
        x, y}$, and, because $\pi$ preserves products and is injective when
        restricted to $S^{w, x, y}$, so is $\pi(S^{w, x, y})$. In that case,
        Proposition \ref{prop:kfree_xy} tells us that $\pi(S^{w, x, y})
        \subseteq
        \mathcal{F}^{xy}$ has relative upper density in
        $\mathcal{F}^{xy}$ at most $1/k$ . Hence, for
        each $(w, x, y) \in \mathcal{I}$, there exists $n_0 = n_0(w, x, y)$ such that
        \begin{equation}
                \label{eq:bound_wxy}
                \frac{\mu(\pi(S^{w, x, y})_{\leq n})}{\mu(\mathcal{F}^{xy}_{\leq n})}
                \leq \frac{1}{k} +
                \varepsilon,
        \end{equation}
        for all $n \geq n_0$.

        Let $N_0$ be large enough so that
        \eqref{eq:bound_wxy} holds for all $w, x, y$ with $|w| \leq \ell_0$, with
        $\ell_0$ to be fixed later on. Our goal is to bound
        \begin{equation}
\label{eq:bound_mu_s}
\frac{\mu(S_{\leq n})}{n} = \sum_{(w, x, y) \in \mathcal{I}} \frac{\mu(S^{w, x, y}_{\leq
n})}{n}= \sum_{\substack{(w, x, y) \in \mathcal{I}, \\ |w| \leq \ell_0}} \frac{\mu(S^{w, x,
y}_{\leq n})}{n} + \sum_{\substack{(w, x, y) \in \mathcal{I}, \\ |w| > \ell_0}} \frac{\mu(S^{w, x, y}_{\leq n})}{n}
        \end{equation}
        for large enough $n \geq N_0$. Let us first bound the sum where $|w| >
        \ell_0$. In that case, applying  \eqref{eq:definition_mu}
        we may deduce the crude bound
        \begin{align*}
                \sum_{\substack{(w, x, y) \in \mathcal{I},\\ |w| > \ell_0}} \frac{\mu(S^{w, x, y}_{\leq n})}{n}
                &= 
                \sum_{\substack{(w, x, y) \in \mathcal{I},\\ |w| > \ell_0}} \frac{1}{n}
                \sum_{i > 2 |w|}^n
                \frac{|S^{w, x, y}(i)|}{|\mathcal{F}(i)|} 
                        \\
                &\leq \sum_{\substack{(w, x, y) \in \mathcal{I},\\ |w| > \ell_0}} \frac{1}{n}
                \sum_{i > 2 |w|}^n
                \frac{|\mathcal{F}(i-2|w|)|}{|\mathcal{F}(i)|} \\
                & \leq \sum_{w \in \mathcal{F}, |w| > \ell_0}(2|\mathcal{A}| -
                1)^{-2|w|}
                \\
                &\leq \sum_{\ell > \ell_0} |\mathcal{F}(l)| (2|\mathcal{A}| -
                1)^{-2\ell}
                \\
                & \leq 2|\mathcal{A}|
                \sum_{\ell > \ell_0} (2|\mathcal{A}| - 1)^{-l} \leq \varepsilon,  
        \end{align*}
        if we take $\ell_0$ large enough in terms of $\varepsilon$. On the other
        hand, note that the first summand of \eqref{eq:bound_mu_s} can be
        rewritten as
        \begin{align*}
                \sum_{\substack{(w, x, y) \in \mathcal{I},\\ |w| \leq \ell_0}}
                \frac{\mu(S^{w, x, y}_{\leq n})}{n} &=
                \sum_{\substack{(w, x, y) \in \mathcal{I}, \\ |w| \leq \ell_0}} \frac{1}{n}
                \sum_{i > 2 |w|}^n
                \frac{|S^{w, x, y}(i)|}{|\mathcal{F}(i)|}
                \\
                &= \sum_{\substack{(w, x, y) \in \mathcal{I}, \\ |w| \leq \ell_0}} \frac{1}{n}
                \sum_{i > 2 |w|}^n
                \frac{|\mathcal{F}(i-2|w|)|}{|\mathcal{F}(i)|}
                \frac{|\pi(S^{w, x, y})(i-2|w|)|}{|\mathcal{F}(i-2|w|)|} \\
                &=  \frac{1}{n}\sum_{\substack{(w, x, y) \in \mathcal{I}, \\ |w|
                \leq \ell_0}}
                \frac{\mu\left(\pi\left(S^{w,x,y}\right)_{\leq
                n-2|w|}\right)}{(2|\mathcal{A}|-1)^{2|w|}}.
        \end{align*}
        Then, using \eqref{eq:bound_wxy}, we obtain the bound
        \begin{align*}
                \sum_{\substack{(w, x, y) \in \mathcal{I},\\ |w| \leq \ell_0}}
                \frac{\mu(S^{w, x, y}_{\leq n})}{n} & \leq
                 \left(\frac{1}{k} +
                 \varepsilon\right) \frac{1}{n}\sum_{\substack{(w, x, y) \in
                 \mathcal{I}, \\ |w| \leq \ell_0}}
                \frac{\mu\left(\mathcal{F}^{xy}_{\leq n-2|w|}\right)}{(2|\mathcal{A}|-1)^{2|w|}}\\
                &\leq \left(\frac{1}{k} +
                \varepsilon\right)\frac{1}{n} \sum_{\substack{(w, x, y) \in
                \mathcal{I}, \\ |w| \leq \ell_0}}
                \mu\left(\mathcal{F}^{w,x,y}_{\leq n}\right) \\
                &\leq \frac{1}{k} + \varepsilon,
        \end{align*}
        for large enough $n$.
        Plugging the bounds for both summands into
        \eqref{eq:bound_mu_s} we obtain that
        \begin{equation*}
                \label{eq:final_bound}
                \bar{d}(S) \leq \frac{1}{k} + 2\varepsilon.
        \end{equation*}
        Since this is true for any $\varepsilon > 0$, we
        conclude by taking $\varepsilon \to 0$.
\end{proof}


\section{Proof over a semigroup}\label{sec:semigroup}
Throughout this section, let $x, y \in \mathcal{A} \cup
\mathcal{A}^{-1}$ be fixed letters such that $x \neq y^{-1}$ and write $\mathcal{G} =
\mathcal{F}^{xy} \subseteq \mathcal{F}$ for the subsemigroup of words
beginning in $x$ and ending in $y$, including the empty word. We also write $M =
\frac{2|\mathcal{A}|}{2|\mathcal{A}| - 1}$ for a constant that appears
in several arguments, due to the fact that
$|\mathcal{F}(i)||\mathcal{F}(j)| = \frac{|\mathcal{F}(i+j)|}{M}$.

We first prove a version of Proposition
\ref{prop:kfree_xy} that depends on the construction of certain subsets of $S$
with appropriate properties.
We say a subset $\mathcal{H} \subseteq \mathcal{F}$ is \emph{dense} if $\mu(\mathcal{H}_{\leq n})
> \delta \mu(\mathcal{F}_{\leq n}) > 0$ for a fixed $\delta$ and all large
enough $n$. Furthermore, it is a \emph{subsemigroup} if $\mathcal{H}\cdot
\mathcal{H} \subseteq \mathcal{H}$, i.e.\ $\alpha \beta \in \mathcal{H}$ for all
$\alpha, \beta \in \mathcal{H}$.
Finally, a subset $W \subseteq \mathcal{H}$ has
\emph{unique products} in $\mathcal{H}$ if the map
\begin{align*}
        (W, \mathcal{H}) &\to \mathcal{H} \\
        (w, h) &\mapsto w \cdot h
\end{align*}
is injective.

For $W$ and $\mathcal{H} \subseteq \mathcal{G}$ satisfying the above
properties, we prove a version of Proposition \ref{prop:kfree_xy} conditional on
$\mu(W)$ being close to $M$.
\begin{lemma}
        \label{lem:unique_implies_bounded}
        Let $\mathcal{H} \subseteq \mathcal{G}$ be  a dense subsemigroup. For any
        strongly $k$-product free set $S \subseteq \mathcal{H}$ and finite
        subset $W \subseteq S$ with unique products in $\mathcal{H}$ it holds that
        \[
                \bar{d}_\mathcal{H}(S) \leq \frac{1}{1 + \frac{\mu(W)}{M} + \dots +
                \left(\frac{\mu(W)}{M}\right)^{k-1}}
        \]
\end{lemma}
\begin{proof}
        For a given $n$, consider a random word $\alpha \in \mathcal{F}_{\leq n}$
        chosen in the following manner. We first choose its length $\ell \in [n]$
        with uniform probabi\-lity, and then choose a random word belonging to
        $\mathcal{F}(\ell)$ uniformly, so that $\Pr (\alpha = w) =
        \mu(\{w\})/n$. In other words, we are sampling $\mathcal{F}_{\leq n}$
        according to the measure $\mu$ normalized so that it is a probability
        measure. By additivity, it holds that
        \begin{equation}
                \label{eq:prob_alpha_h}
                \frac{\Pr(\alpha \in S)}{\Pr(\alpha \in \mathcal{H})} =
                \frac{\mu(S_{\leq n})}{\mu(\mathcal{H}_{\leq n})}
        \end{equation}

        Since $S$ is strongly $k$-product-free, if $\alpha$ belongs to $W^i
        \cdot S$ for $1 \leq i \leq k-1$ then $\alpha$ cannot belong to $S$.
        Since $W$ has unique products in $\mathcal{H}$, the set $W^i \cdot S$ intersects $W^j
        \cdot S$ with $i > j$ only when $W^{i-j} \cdot S$ intersects $S$. Hence,
        the sets $W^i \cdot S$ for $0 \leq i \leq k-1$ are disjoint and
        \begin{equation}
                \label{eq:disjoint_sets_kprod}
                \sum_{i=0}^{k-1}\Pr (\alpha \in W^i \cdot S)
                \leq \Pr(\alpha \in \mathcal{H}).
        \end{equation}

        Let us estimate the probability that $\alpha \in W \cdot A$ for an
        arbitrary subset $A \subseteq \mathcal{H}$. Note that
        $\alpha$ factors in a unique way because $W$ has unique products. In
        particular, $W(i) \cdot S$ and $W(j) \cdot S$ are disjoint for $i \neq
        j$. Hence,
        writing $\ell_0$ for the maximum length of a word in
        $W$, we have that
        \[
                \Pr (\alpha \in W \cdot A) = \sum_{i=1}^{\ell_0} \Pr (\alpha \in
                W(i) \cdot A).
        \]
        Conditioning on the length of $\alpha$, we obtain
        \begin{equation}
                \label{eq:product_WA}
        \begin{split}
                \Pr (\alpha \in W \cdot A) &=
                \sum_{i=1}^{\ell_0} \sum_{l=i+1}^n \Pr(|\alpha| = l)
                \Pr \left(\alpha \in W(i) \cdot A \bigmid |\alpha| = \ell\right) \\
                &= \sum_{i=1}^{\ell_0} \sum_{l=i+1}^n \frac{1}{n}
                \Pr \left(\alpha \in W(i) \cdot A(\ell-i) \bigmid |\alpha| =
                \ell\right) \\
                &= \sum_{i=1}^{\ell_0} \sum_{l=i+1}^n \frac{1}{n}
                \mu(W(i))
                \mu(A(\ell-i))
                \frac{|\mathcal{F}(i)| |\mathcal{F}(\ell-i)|}{|\mathcal{F}(\ell)|}\\
                &= \frac{1}{M}\sum_{i=1}^{\ell_0} \mu(W(i))
                \left(\frac{1}{n}\sum_{t=1}^{n-i} \mu(A(t))\right)\\
                &= \frac{\mu(W)}{M} \frac{\mu(A_{\leq n}) + O(1)}{n} \\
                &= \frac{\mu(W) \Pr(\alpha \in A)}{M} + o(1),
        \end{split}
\end{equation}
        where in the third equality we used that there is no cancellation in
        $\mathcal{G}$, so that $|W(i) \cdot A(\ell-i)| = |W(i)| \cdot |A(\ell-i)|$.
        By induction, we have that
        $\Pr(W^i \cdot S) = (\mu(W)/M)^i \Pr(\alpha \in S) + o(1)$. Plugging
        this estimate into
        \eqref{eq:disjoint_sets_kprod} and noting that $\mathcal{H}$ is dense gives
        \[
                \frac{\Pr(\alpha \in S)}{\Pr(\alpha \in \mathcal{H})} \left(1 +
                        \frac{\mu(W)}{M}
        + \cdots + \left(\frac{\mu(W)}{M}\right)^{k-1}\right) + o(1) \leq 1
        \]
        Using \eqref{eq:prob_alpha_h} and taking the upper limit we reach
        the desired conclusion.
\end{proof}

We would now like to find a subset $W \subseteq S$ large enough in order to
apply the previous lemma. We first present a proof for the case $2 \leq k \leq
3$, where we can use a similar argument to the one in \cite{2020.LLNW}. In this
case, we may exploit the fact that $S$ is product-free to construct $W$ in a
straightforward manner. The reader interested only in the general case may
choose to ignore this proof.


\begin{proof}[Proof of Proposition \ref{prop:kfree_xy} for $2 \leq k \leq 3$]
        Suppose there exists a $k$-product-free subset $S$ satisfying
        $\bar{d}_{\mathcal{G}}(S) > \frac{1}{k} + \varepsilon$ for $\varepsilon
        > 0$. Our goal is to build a subset $W \subseteq S$ with unique products
        in $\mathcal{G}$ and $\mu(W)$ as large as possible so that we may apply
        Lemma \ref{lem:unique_implies_bounded} with $\mathcal{H} = \mathcal{G}$.
        We will build a nested sequence of divisor-free subsets $W_i \subseteq
        S$, meaning there are no $u, v \in W_i$ and $t \in \mathcal{G}$ such
        that $ut = v$. This condition implies that $W_i$ has unique products in
        $\mathcal{G}$.

        Set $W_1 = S(\ell)$ for
some $\ell$ such that $S(\ell) \neq \varnothing$. Given $W_i \subseteq S$, we will now
        build $W_{i+1}$.
        Choose $\alpha \in
        \mathcal{F}_{\leq n}$ randomly as in the proof of Lemma
        \ref{lem:unique_implies_bounded}. Since $S$ is strongly
        $k$-product-free, it holds that $(W_i^j \cdot S) \cap S = \varnothing$ for
        $1 \leq j \leq k-1$. Since $W_i$ has unique products in $\mathcal{G}$,
        this implies that $(W^{t+j}_i \cdot S) \cap (W_i^t \cdot S) =
        \varnothing$ for $1 \leq j \leq k-1$ and $t \geq 0$.
        Therefore, the subsets $W_i^j \cdot S$ are disjoint for $0 \leq j \leq k-1$ and
        \[
                \sum_{j = 0}^{k-1} \Pr(\alpha \in W_i^j \cdot S) \leq \Pr(\alpha \in W_i \cdot
                \mathcal{G}) + \Pr(\alpha \in S, \alpha \not \in W_i \cdot
                \mathcal{G}).
        \]
        Repeatedly applying the estimate \eqref{eq:product_WA} and reordering
        terms, we obtain that
        \[
                \Pr(\alpha \in S, \alpha \not \in W_i \cdot
                \mathcal{G}) \geq
                \Pr(\alpha \in S) \left( \sum_{j=0}^{k-1} \left(\frac{\mu(W_i)}{M}\right)^j\right) -
                \Pr(\alpha \in \mathcal{G})\frac{\mu(W_i)}{M} + o(1)
        \]
        Since we assumed that $S$ has relative upper density greater than
        $1/k + \varepsilon$, we deduce that
        \[
                \Pr(\alpha \in S, \alpha \not \in W_i \cdot
                \mathcal{G}) \geq
                \Pr(\alpha \in G) \left( \frac{1}{k} \left(\sum_{j=0}^{k-1}
                \left(\frac{\mu(W_i)}{M}\right)^j\right) -
        \frac{\mu(W_i)}{M}\right)
        \]
        for large enough $n$. Since it holds that $\Pr(\alpha \in A) =
        \frac{1}{n}\sum_{i=1}^n \mu(A(i))$, we may find $\ell \in \mathbb{N}$ such
        that
        \[
                \mu(S(\ell) \setminus W_i \cdot \mathcal{G}) \geq \Pr(\alpha
                \in \mathcal{G})\left( \frac{1}{k}
        \left( \sum_{j=0}^{k-1} \left(\frac{\mu(W_i)}{M}\right)^j\right)
- \frac{\mu(W_i)}{M}\right).
        \]
        Set $W_{i+1} = W_i \sqcup (S(\ell) \setminus W_i \cdot \mathcal{G})$, and
        note that is still divisor-free and $W_{i+1} \subseteq S$.

        With this procedure, we obtain a sequence of subsets satisfying
        \[
                \mu(W_{i+1}) - \mu(W_i) \geq \Pr(\alpha \in \mathcal{G})\left(\frac{1}{k} \left(\sum_{j=0}^{k-1} \left(\frac{\mu(W_i)}{M}\right)^j\right) - \frac{\mu(W_i)}{M}\right).
        \]
        Let us now study the limit $\lim_{i \to \infty} \mu(W_i)$. Define the
        polynomial
        \begin{equation}
                \label{eq:polynomial_prodfree}
                f(x) = \Pr(\alpha \in \mathcal{G})\left(\frac{1}{k}
                \left(\sum_{j=0}^{k-1} \left(\frac{x}{M}\right)^j\right) -
        \frac{x}{M}\right),
\end{equation}
        and the sequence $x_i = \mu(W_i)$, so that we have $x_{i+1} - x_i \geq
        f(x_i)$. Since the sequence $x_i$ is increasing and bounded, it has a
        limit $r = \lim_{i \to \infty} x_i$, which, by continuity of $f$, must
        satisfy $0 = r - r \geq f(r)$. For $2 \leq k \leq 3$, the polynomial $f$ is
        positive at zero and has its first positive zero at $M$. Hence, for $2
        \leq k \leq 3$ we obtain that $r = \lim_{i \to \infty} \mu(W_i) \geq M$.

        Applying Lemma \ref{lem:unique_implies_bounded} with every $W_i$
        and letting $i \to \infty$ we obtain $\bar{d}_{\mathcal{G}}(S) \leq
        1/k$, a contradiction with our original assumption.
\end{proof}

The previous proof fails for $k > 3$ because the polynomial in
\eqref{eq:polynomial_prodfree} has a zero in the interval $(0,M)$. Thus, we
must find another way of building the subset $W \subseteq S$ necessary to apply Lemma
\ref{lem:unique_implies_bounded}. To do so, instead of finding such a set
directly in $\mathcal{G}$, we will find a subset of $\mathcal{G}$
where $S$ is regularly distributed, where the existence of $W \subseteq S$ as we
are interested in is much easier to prove and does not depend on $S$ being
product-free.

Concretely, we are interested in studying $S$ when we restrict ourselves to
words which are divisible in $\mathcal{G}$ by a given factor. For a given $w
\in \mathcal{G}$, we write $w\mathcal{G} \subseteq \mathcal{G}$ for the set
of words belonging to $\mathcal{G}$ which may be written as $w \alpha$ for
$\alpha \in \mathcal{G}$. We then
define the following {pseudorandomness} condition, which measures whether
$S$ is evenly distributed when restricted to such sets.
\begin{definition}
        Given $w \in \mathcal{G}$, a subset $S \subseteq w\mathcal{G}$ is
        $\varepsilon$-regular in $w\mathcal{G}$ if
        \begin{equation}
                \label{eq:pseudorandomness}
                \bar{d}_{ww'\mathcal{G}}((S \cap (ww'\mathcal{G}))  \leq
                \bar{d}_{w\mathcal{G}}(S) +
                \varepsilon
        \end{equation}
        for all words $w' \in \mathcal{G}$.
\end{definition}

We now prove the analogous statement to Theorem
\ref{theo:kfree} under pseudorandomness assumptions. In particular, the
following lemma will imply Theorem \ref{theo:kfree} when $S$ is $\varepsilon$-regular for all
$\varepsilon > 0$.
\begin{lemma}
        \label{lemma:kprod_pseudorandom}
        Let $S \subseteq w\mathcal{G}$ be a strongly $k$-product free set that is
        $\varepsilon$-regular in $w\mathcal{G}$, with $w \in \mathcal{G}$, and
        let $d = \bar{d}_{w\mathcal{G}}(S)$ be its relative upper density. Then
        \begin{equation}
\label{eq:kprod_pseudorandom}
                d \left(1 + \frac{d}{d+2\varepsilon} + \cdots +
                \left(\frac{d}{d+2\varepsilon}\right)^{k-1} \right) \leq 1,
        \end{equation}
\end{lemma}
\begin{proof}
        Our goal is to build a subset $W \subseteq S$ with unique
        products in $w\mathcal{G}$ and $\mu(W)$ as large as possible so that we may apply Lemma
        \ref{lem:unique_implies_bounded} with $\mathcal{H} = w\mathcal{G}$. We
        will construct a nested sequence of subsets $W_i \subseteq S$
        with the following properties:
        \begin{enumerate}[i)]
                \item Elements of $W_i$ have either the same length or lengths
                        that differ by at least $|w|$. That is, given $w_1, w_2
                        \in W_i$, either $|w_1| = |w_2|$ or $\big||w_1| -
                        |w_2|\big| >
                        |w|$.
                \item The sets $W_i$ are divisor-free in $w\mathcal{G}$. By this we
                        mean that there are no $u, v \in W_i$ and $t \in w\mathcal{G}$
                        such that $ut = v$.
        \end{enumerate}
        From these two properties, we may deduce that the sets $W_i$ have unique
        products in $w\mathcal{G}$. Indeed, if $ux = u'x'$ for some $u, u' \in
        W_i$ and $x, x' \in w\mathcal{G}$, then $|u| = |u'|$ implies $u=u'$ and
        $x=x'$, and it cannot be that $|u| < |u'| - |w|$ because
        $W_i$ is divisor-free.
        Set $W_1 = S(\ell)$ for some $\ell$ such that $S(\ell) \neq \varnothing$.
        Given $W_i$, we will now build $W_{i+1}$. For a given $n$,
        choose $\alpha$ randomly as in the proof of Lemma
        \ref{lem:unique_implies_bounded}.
        We have that
        \begin{equation}
\label{eq:lower_bound_wk}
                \Pr(\alpha \in S,\, \alpha \not \in W_i \cdot w\mathcal{G})
                \geq \Pr(\alpha \in S) - \Pr(\alpha \in S, \alpha \in W_i \cdot
                w\mathcal{G}).
        \end{equation}
        Let us upper bound the rightmost term. Since $W_i$ has unique products,
        \begin{align*}
                \Pr(\alpha \in S, \alpha \in W_i \cdot
                w\mathcal{G}) &= \sum_{wu \in W_i} \Pr(\alpha \in S, \alpha \in
                wuw \mathcal{G}) \\
                &= \sum_{wu \in W_i} \frac{\mu((wuw\mathcal{G})_{\leq
                n})}{\mu(\mathcal{F}_{\leq n})}  \frac{\mu((S \cap
                (wuw \mathcal{G}))_{\leq n})}{\mu((wuw\mathcal{G})_{\leq
                n})}.
        \end{align*}
        Consider a fixed $\varepsilon > 0$.
        By the $\varepsilon$-regularity of $S$, using
        \eqref{eq:pseudorandomness} we may guarantee for all large enough $n$
        that $\frac{\mu((S \cap (wuw\mathcal{G}))_{\leq
        n})}{\mu((wuw\mathcal{G})_{\leq n})} \leq d + \varepsilon$, from which we
        conclude that
        \[
                \Pr(\alpha \in S, \alpha \in W_i \cdot w\mathcal{G}) \leq
                \Pr(\alpha \in W_i \cdot w \mathcal{G}) (d + \varepsilon).
        \]
        Using the estimate \eqref{eq:product_WA} we obtain
        \[
                \Pr(\alpha \in S, \alpha \in W_i \cdot w\mathcal{G}) \leq \frac{\mu(W_i)
                \Pr(\alpha \in w \mathcal{G})}{M}(d + \varepsilon) + o(1)
        \]

        Plugging this into \eqref{eq:lower_bound_wk}, for any $\xi > 0$ we may
        obtain arbitrarily large $n$ such that
        \begin{align*}
                \Pr(\alpha \in S,\, \alpha \not \in W_i \cdot
                w\mathcal{G})
        &\geq\Pr(\alpha \in w\mathcal{G}) \left(d - \xi - \frac{\mu(W_i)}{M}
        (d + \varepsilon)\right).
        \end{align*}
        Setting $\xi = \varepsilon \mu(W_i) / M$ and using the fact that
        $\Pr(\alpha \in A) = \frac{1}{n}\sum_{i=1}^n \mu(A(i))$, we may obtain
        $\ell \in \mathbb{N}$ as large as necessary such that
        \[
                \mu(S(\ell) \setminus W_i \cdot w\mathcal{G}) \geq \Pr(\alpha \in
                w\mathcal{G}) \left(d -
                \frac{\mu(W_i)}{M}(d+2\varepsilon)\right).
        \]
        Taking $\ell$ larger than the maximum length of an element in
        $W_i$ plus $|w|$ and setting $W_{i+1} = W_i \sqcup (S(\ell) \setminus W_i
        \cdot w\mathcal{G})$ we obtain a subset $W_{i+1} \subseteq S$ with the desired
        properties.

        This gives us a nested sequence of subsets $W_i$ satisfying
        \[
                \mu(W_{i+1}) \geq \mu(W_i) + \Pr(\alpha \in w\mathcal{G})\left(d -
                \frac{\mu(W_i)}{M}(d+2\varepsilon)\right).
        \]
        We analyse the limit $\lim_{i \to \infty} \mu(W_i)$ as we did in the
        proof of Proposition \ref{prop:kfree_xy}.
        Define $x_i = \mu(W_i)$ and the linear function
        \[
                f(x) = \Pr(\alpha \in w\mathcal{G})\left(d -
                \frac{x}{M}(d+2\varepsilon)\right),
        \]
        so that we have $x_{i+1} - x_i \geq
        f(x_i)$. Since the sequence $x_i$ is increasing and bounded, it has a
        limit $r = \lim_{i \to \infty} x_i$, which, by continuity of $f$, must
        satisfy $0 = r - r \geq f(r)$. Since $f$ is positive at zero and its
        first positive zero is at $x = \frac{Md}{d+2\varepsilon}$, we obtain that
        $r = \lim_{i \to \infty} \mu(W_i) \geq \frac{Md}{d+2\varepsilon}$.
        Applying Lemma \ref{lem:unique_implies_bounded} with every $W_i$ and
        letting $i \to \infty$ gives the desired conclusion.
\end{proof}

Finally, we use a density-increment strategy, where failure of
pseudorandomness implies an increase in density, to prove our main result.
\begin{proof}[Proof of Proposition \ref{prop:kfree_xy}]
        Consider a fixed $\varepsilon > 0$. If $S$ is not $\varepsilon$-regular
        in $\mathcal{G}$,
        we may then find $w_1 \in \mathcal{G}$ such that
        \[
                \bar{d}_{w_1\mathcal{G}}(S \cap w_1\mathcal{G}) >
                \bar{d}_{\mathcal{G}}(S) +
                \varepsilon.
        \]
        More generally, if $S \cap w_k\mathcal{G}$ is not $\varepsilon$-regular in
        $w_k\mathcal{G}$, we may find $w_{k+1} \in w_k \mathcal{G}$ such that
        \[
                \bar{d}_{w_{k+1}\mathcal{G}}(S \cap w_{k+1}\mathcal{G}) >
                \bar{d}_{w_k \mathcal{G}}(S \cap w_k \mathcal{G}) +
                \varepsilon > \bar{d}_{\mathcal{G}}(S) +
                k\varepsilon.
        \]
        Since relative upper density is bounded above by $1$, this procedure
        must terminate in less than $\varepsilon^{-1}$ steps, which means there
        exists $w$ such that $S \cap w\mathcal{G}$ is $\varepsilon$-regular in
        $w\mathcal{G}$ and $\bar{d}_{w\mathcal{G}}(S \cap w\mathcal{G}) \geq
        \bar{d}(S)$.

        Applying Lemma \ref{lemma:kprod_pseudorandom} with $S \cap
        w\mathcal{G}$ and using that the left
        hand side in \eqref{eq:kprod_pseudorandom} is increasing in $d$ we obtain that
        \[
                \bar{d}(S) \left(1 + \frac{\bar{d}(S)}{\bar{d}(S)+2\varepsilon} + \cdots +
                \left(\frac{\bar{d}(S)}{\bar{d}(S)+2\varepsilon}\right)^{k-1} \right) \leq 1.
        \]
        Taking $\varepsilon \to 0$ finishes the proof.
\end{proof}

\section{Final remarks}\label{sec:final}
We devoted the previous section to proving Proposition \ref{prop:kfree_xy},
which is stated for the semigroup $\mathcal{F}^{xy}$ for $x, y \in \mathcal{A}
\cup \mathcal{A}^{-1}$ and $x \neq y^{-1}$. 
In fact, the proof of Theorem \ref{thm:semigroup} is done following the arguments from the previous section, by only replacing
the role played by $\mathcal{F}^{xy}$ for $\bm{F}_{\mathcal{A}}$, the free
semigroup over $\mathcal{A}$, and by replacing the constant $M$ by 1. It is also worth noting
that the results in \cite{2020.LLNW} concern the upper Banach density of product-free subsets, which gives slightly stronger results,
since the upper Banach density is an upper bound for the upper asymptotic density we
consider. For the sake of simplicity, we have not attempted to write down our
results for this case, although all arguments should hold.

Finally, it would also be interesting to consider the case of $k$-product-free
sets. To state the natural conjecture for this case, define $\rho$ as
\[
        \rho(k) = \min\left( \left\{ \ell \in \mathbb{N} \colon \ell \nmid k-1
        \right\} \right).
\]
Then we believe the following to be true.
\begin{conjecture}
        \label{conj:kprod}
        Let $S \subseteq \mathcal{F}$ be a $k$-product-free subset for
        $k \geq 2$. Then
        \begin{equation}
                \bar{d}(S) \leq \frac{1}{\rho(k)}
        \end{equation}
\end{conjecture}
This is analogous to a result of \L uczak and Schoen \cite{1997.LS}, which
proves the corresponding statement over the integers. One might attempt to prove
it using ideas by the same authors in \cite{2001.LS}. For example, for $k=3$,
one would partition $S$ into $S_1$ and $S_2$, with $w \in S_2$ if $w = uv$ for
$u, v \in S$ and $w \in S_1$ otherwise. Then one could use the following
disjunctive for an extremal $S$:
\begin{itemize}
        \item If $\bar{d}(S_2) > 0$, we would expect to find a large enough
                subset $W \subseteq S_2$ with unique products. One could then
                directly use $W$ as done in Lemma
                \ref{lem:unique_implies_bounded} because $(S_2 \cdot S) \cap S =
                \varnothing$.
        \item If $\bar{d}(S_2) = 0$, then $\bar{d}(S_1) = \bar{d}(S)$ and $S_1$
                is strongly $3$-product-free.
\end{itemize}
This approach fails because of the claim in the first point. It is not clear how
to find such a $W$. If we try to use the same approach as we did in this paper,
the control implied by our definition of pseudorandomness only allows for upper
bounds and not lower bounds on the density of $S$, and this complicates
controlling the regularity of $S$ and $S_2$ at the same time, which is
what we would need to find a suitable $W$. For stronger notions of
pseudorandomness, the density increment strategy seems to break.

Note added: \emph{After the preparation of our paper and the
        publication of our first preprint, Illingworth, Michel and Scott
        \cite{2023.IMS} proved Conjecture
        \ref{conj:kprod}. In their paper, they also prove
further results which characterise extremal product-free sets in the free
semigroup, solving a
conjecture from \cite{2020.LLNW}.}

\section*{Acknowledgments}

We acknowledge the support of the grants MTM2017-82166-P, PID2020-113082GB-I00 funded by MICIU/AEI/10.13039/501100011033,
and the Severo Ochoa and Mar\'ia de Maeztu Program for Centers and Units of
Excellence in R\&D (CEX2020-001084-M). Miquel Ortega also acknowlegdes the
support of the FPI grant PRE2021-099120.

\bibliographystyle{chicago}
\bibliography{kproductfree}
\end{document}